\documentclass[amsthm]{elsart}

\usepackage{yjsco}
\usepackage{algorithm}
\usepackage{algorithmic}
\usepackage{amsmath}
\usepackage{amssymb}

\newcommand{\Gr}{Gr\"obner }

\def\LM{{\mathrm{LM}}}
\def\LC{{\mathrm{LC}}}
\def\LT{{\mathrm{LT}}}

\def\poly{{\mathrm{poly}}}
\def \LT{{\rm LT}}
\def \LM{{\rm LM}}
\def \LC{{\rm LC}}

\def \lcm{{\rm lcm}}

\def\ri{\rangle}
\def\li{\langle}
\def\LM{{\mathrm{LM}}}
\def\LC{{\mathrm{LC}}}
\def\LT{{\mathrm{LT}}}
\def\L{{\mathcal{L}}}

\def\nf{{\rm{NF}}}
\def\poly{{\rm{poly}}}
\def\anc{{\rm{anc}}}

\begin{document}
\begin{frontmatter}

\title{A Variant of Gerdt's Algorithm for Computing Involutive Bases}


\author{Vladimir P. Gerdt}
\address{Laboratory of Information Technologies\\Joint Institute for Nuclear Research, 141980 Dubna, Russia}
\ead{Gerdt@jinr.ru}
\ead[url]{http://compalg.jinr.ru/CAGroup/Gerdt/}

\author{Amir Hashemi}
\address{Department of Mathematical Sciences, \\Isfahan University of Technology, Isfahan, 84156-83111, Iran}
\ead{Amir.Hashemi@cc.iut.ac.ir}
\ead[url]{www.amirhashemi.iut.ac.ir}

\author{Benyamin M.-Alizadeh}
\address{Department of Mathematical Sciences, \\Isfahan University of Technology, Isfahan, 84156-83111, Iran}
\ead{B.Alizadeh@math.iut.ac.ir}

\begin{abstract}
In \cite{gerdtnew}, the first author presented an efficient algorithm for computing involutive (and reduced Gr\"obner) bases. In this paper, we consider a modification of this algorithm which simplifies matters to understand it and to implement. We prove correctness and termination of the modified algorithm and also correctness of the used criteria. The proposed algorithm has been implemented in {\tt Maple}. We present experimental comparison, via some examples, of performance of the modified algorithm with its original form described in \cite{gerdtnew} and has been implemented in {\tt Maple} too. In doing so, we have taken care to provide uniform implementation details for the both algorithms.
\end{abstract}


\end{frontmatter}

\section{Introduction}

The most important algorithmic object in computational algebraic geometry is {\em Gr\"{o}bner basis}. The notion of \Gr basis was introduced and an algorithm of its construction was designed in 1965 by Buchberger in his Ph.D. thesis \cite{Bruno1}. Later one \cite{Bruno2}, he discovered two criteria for detecting some unnecessary, and thus useless, reductions that made the Gr\"{o}bner bases method a practical tool to solve a wide class of problems in polynomial ideals theory and in many other research areas of science and engineering~\cite{Bruno3}.

In the 20s of the last century, French mathematician Janet \cite{janet} developed a constructive approach to analysis of certain systems of partial differential equations based on their completion to involution (cf.  \cite{seiler}). The Janet approach was generalized to arbitrary polynomially nonlinear systems by American mathematician Thomas \cite{thomas}. Based on the related methods described in the book by Pommaret~\cite{pommaret}, Zharkov and Blinkov \cite{zharkov} introduced the notion of {\em involutive bases} in commutative algebra. The particular form of an involutive basis they used is nowadays called {\em Pommaret basis} \cite{gerdt0,seiler}.

 Gerdt and Blinkov \cite{gerdt0} have proposed a more general concept of involutive bases for polynomial ideals and designed algorithmic methods for their construction. The underlying idea of the involutive approach is to translate the methods originating from Janet's results into polynomial ideals theory in order to construct involutive bases by combining algorithmic ideas of the theory of \Gr bases and those in the theory of involutive differential systems. In doing so, Gerdt and Blinkov \cite{gerdt0} introduced the concept of {\em involutive division}. Moreover, they have applied the involutive form of Buchberger's criteria. This led to a strong computational tool which is a serious alternative to the conventional Buchberger algorithm (note that any involutive basis is also a Gr\"obner basis). Apel and Hemmecke in \cite{detecting} proposed two more criteria for detecting unnecessary reductions in involutive basis computations (see also \cite{gerdt2}) which, in the aggregate with the criteria by Gerdt and Blinkov \cite{gerdt0}, are equivalent to Buchberger's criteria. The first author in \cite{gerdtnew} described an efficient algorithm to compute involutive and Gr\"obner bases using all these criteria. We refer to Seiler's book \cite{seiler} for a comprehensive study and application of involution to commutative algebra and geometric theory of partial differential equations.

In this paper, we propose a variant of Gerdt's algorithm \cite{gerdtnew} for constructing involutive bases. This algorithm seems to be simpler both to understand and to implement. We prove correctness for the new version of algorithm which includes the criteria and its termination. We have  implemented this algorithm in {\tt Maple} and we compare here its performance with our implementation of the original Gerdt's algorithm via some examples.

The structure of the paper is as follows. Section \ref{sec:1} contains the basic definitions and notations related to theory of involutive bases, and a short description of the algorithm for their computing proposed in its initial form \cite{gerdt0}. In Section \ref{gerdt}, we briefly present Gerdt's algorithm from \cite{gerdtnew} to compute involutive bases. Section \ref{sec:3} is devoted to description of the modified algorithm together with a proof of its correction and termination. In Section \ref{Exp}, we compare, by using some benchmarking examples, performance of the last algorithm with our implementation of Gerdt's algorithm done also in {\tt Maple}.

\section{Preliminaries}
\label{sec:1}
In this section, we recall some basic definitions and notations from the theory of involutive bases which are used in the paper, and briefly describe the initial algorithm proposed in \cite{gerdt0} for computing such bases.

Let $K$ be a field and $R=K[x_1,\ldots,x_n]$ the polynomial ring in the variables $x_1,\ldots,x_n$ over $K$. Below, we denote a monomial $x_1^{\alpha_1}\cdots x_n^{\alpha_n}\in R$ by ${\bf x}^\alpha$ where $\alpha=(\alpha_1,\ldots,\alpha_n) \in \mathbb{N}^{n}$ is a sequence of non-negative integers. A {\em monomial ordering} on $R$ used in theory of \Gr and involutive bases is a total order $\prec$ on the set of all monomials satisfying the following two properties:

\begin{itemize}
  \item  It is multiplicative; i.e., ${\bf x}^\alpha \prec {\bf x}^\beta$ implies ${\bf x}^{\alpha+\gamma} \prec {\bf x}^{\beta+\gamma}$ for all $\alpha,\beta,\gamma \in \mathbb{N}^{n}$,
\item The constant monomial is the smallest; i.e., $1\prec {\bf x}^\alpha$ for all $\alpha \in  \mathbb{N}^{n}$.
\end{itemize}

A typical  example of a such monomial ordering is the pure lexicographical ordering, denoted by $\prec_{lex}$. For two monomials ${\bf x}^\alpha,{\bf x}^\beta\in R$ ${\bf x}^\alpha \prec_{lex} {\bf x}^\beta$ if the left most nonzero entry of $\beta-\alpha$ is positive.

Let  $I=\langle f_1,\ldots ,f_k\rangle$ be the ideal of $R$ generated by the polynomials $f_1,\ldots ,f_k\in R$. Also let $f\in R$ and $\prec$ be a monomial ordering on $R$. The {\em leading monomial} of $f$ is the greatest monomial (with respect to $\prec$) appearing in $f$, and we denote it by $\LM(f)$. If $F\subset R$ is a set of polynomials, we denote by $\LM(F)$ the set $\{\LM(f) \ \mid \ f\in F\}$. The {\em leading coefficient} of $f$, denoted by $\LC(f)$, is the coefficient of $\LM(f)$. The {\em leading term} of $f$ is $\LT(f)=\LC(f)\LM(f)$. The {\em leading term ideal} of $I$ is defined to be
$$\LT(I)=\langle \LT(f)\ \mid \ f \in I\rangle.$$
A finite set $G=\{g_1,\ldots ,g_k\}\subset I$ is called a {\em Gr\"obner basis} of $I$ w.r.t. $\prec$ if $\LT(I)=\langle \LT(g_1),\ldots,\LT(g_k) \rangle$. For more details, we refer to \cite{Becker}, pages 213--214.

We recall below the definition of a special kind of involutive bases, namely, {\em Janet basis}. For this purpose, we describe first the notion of an {\em involutive division} \cite{gerdt0} which is a restricted monomial division \cite{gerdtnew} that, together with a monomial ordering, determines the properties and the structure of an involutive basis. This makes the main difference between involutive bases and Gr\"obner bases. The idea is to partition the variables into two subsets of multiplicative and non-multiplicative variables, and only the multiplicative variables can be used in the divisibility relation.

\begin{defn}(\cite{gerdt0})
An involutive division $\L$ on the set of monomials of $R$ is given, if  for any finite set $U$ of monomials and any $u \in U$, the set of variables is partitioned into the subsets of multiplicative $M_{\L}(u,U)$ and non-multiplicative $NM_{\L}(u,U)$ variables, and the following three conditions hold:
\begin{itemize}
\item $u,v\in U,\ u{\L}(u,U) \cap v{\L}(v,U) \ne \emptyset \Longrightarrow u\in v{\L}(v,U)$ or $v \in u{\L}(u,U)$
\item $v\in U,\ v \in u\L(u,U) \Longrightarrow \L(v,U) \subset \L(u,U)$
\item $u \in V$ and $V \subset U \Longrightarrow \L(u,U) \subset \L(u,V)$
\end{itemize}
where $\L(u,U)$ denotes  the monoid generated by the variables in $M_\L(u,U)$. If $v \in u\L(u,U)$ then  $u$ is called $\L-${\em (involutive) divisor} of $v$ and the involutive divisibility relation is denoted by $u |_{\L} v$.
If $v$ has no involutive divisors in a set $U$, then it is called ${\L}-$irreducible modulo $U$.
\end{defn}

There are involutive divisions based on the classical partitions of variables suggested by  Janet \cite{janet} and Thomas \cite{thomas}. In this paper, we are interested only in Janet division \cite{gerdt0}.
\begin{exmp}(Janet division)
Let $U$ be a finite set of monomials and $u \in U$. The variable $x_1$ is multiplicative for $u$, if $u$ has the maximum degree of $x_1$ in $U$. For $i>1$, the variable $x_i$ is multiplicative for $u$ if it has the maximum degree of $x_i$ in the set $\{v\in U \ | \  \deg_j(v)=\deg_j(u), 1\leq j <i\}$, where $\deg_j(v)$ denotes the degree of $x_j$ in a monomial $v$.
\end{exmp}

\begin{prop} {\em (\cite{gerdt0})}
\label{jan}
 Janet division satisfies the conditions in Definition 1 and is noetherian.
\end{prop}

{\em Throughout this paper ${\L}$ is assumed to be the Janet division}. Now, we define an involutive basis.

\begin{defn} (\cite{gerdtnew})
Let $I\subset R$ be an ideal and $\prec$ a monomial ordering on $R$. Let ${\L}$ be an involutive division. A finite set $G\subset I$ is an (${{\L}}$-)involutive basis for $I$ if for all $f\in I$ there exists $g\in G$ such that $\LM(g) |_{\L} \LM(f)$.
\end{defn}

From this definition and from that for \Gr basis~\cite{Bruno1,Becker} it follows that an involutive basis for an ideal is its Gr\"obner basis, but the converse is not always true.

\begin{rem}
By using an involutive division in the conventional division algorithm for polynomial rings, we obtain an involutive division algorithm. If $F$ is a finite polynomial set and ${\L}$ is an involutive division, then we use $\nf_{\L}(f,F)$ to denote the remainder of $f$ on involutive division by $F$.
\end{rem}

The following theorem provides an algorithmic characterization of involutive basis for a given ideal which is an involutive analogue of the Buchberger characterization of \Gr basis.
\begin{thm} {\em (\cite{gerdt0})}
\label{blin}
Let $I\subset R$ be an ideal, $\prec$ a monomial ordering on $R$, and ${\L}$ an involutive division. Then a finite subset $G \subset I$ is an involutive basis for $I$ if for each $f\in G$ and each $x\in NM_{\L}(\LM(f),\LM(G))$, we have $\nf_{\L}(xf,G)=0$.
\end{thm}

Based on this theorem, one can design an algorithm to compute involutive bases. We recall here the {\sc InvBasis} algorithm from \cite{gerdt0} (see also \cite{gerdtnew}) which computes a {\em minimal} involutive basis for an ideal. An involutive basis $G$ is called minimal if for any other involutive basis $\tilde{G}$ the inclusion $\LM(G)\subseteq \LM(\tilde{G})$ holds.  A minimal involutive basis exists and being monic and {\em involutively autoreduced}, i.e. satisfying
\[
(\,\forall g\in G\,)\ \  [\,g=\nf_{{\L}}(g,G\setminus\{g\})\,]
\]
is {\em unique} for a given ideal, a monomial ordering and a {\em constructive} involutive division (see \cite{gerdt0} for the definition of constructivity).

\begin{algorithm}[H]
\caption{{\sc InvBasis}}
\begin{algorithmic}
 \INPUT $F$, a set of polynomials; ${\L}$, an involutive division;, $\prec$, a monomial ordering.
 \OUTPUT a minimal involutive basis for $\li F \ri$ w.r.t ${\L}$ and $\prec$
 \STATE Select $f \in F$ with no proper divisor of $\LM(f)$ in $\LM(F)$
\STATE $G:=\{f\}$;
\STATE $Q:=F \setminus G$;
\WHILE {$Q \ne \emptyset$}
     \STATE Select and remove $p \in Q$ with no proper divisor of $\LM(p)$ in $\LM(Q)$;
     \STATE $h:=\nf_{\L}(p,G)$;
     \IF {$h\ne 0$}
           \FOR {$g\in G$ which $\LM(g)$ is properly divisible by $\LM(h)$}
                  \STATE {$Q:=Q\cup \{g\}$; }
		  \STATE $G:=G \setminus \{g\}$;
           \ENDFOR
           \STATE {$G:=G\cup \{h\}$;}
           \STATE {$Q:=Q\cup \{xg \ |\ g\in G, x\in NM_{\L}(\LM(g),\LM(G))\}$;}
     \ENDIF
\ENDWHILE
\STATE {\bf Return} ($G$)
\end{algorithmic}
\end{algorithm}

Here, as it has been mentioned in \cite{gerdtnew}, in comparison to the algorithm in \cite{gerdt0} another selection strategy is used. By this strategy, a polynomial whose leading monomial has no proper divisor is chosen. However, this algorithm is not efficient in practice, it processes the repeated prolongations and does not use any criterion to avoid the unnecessary reductions. For this purpose, we consider in the next section, the Gerdt' algorithm from \cite{gerdtnew} as an improvement of {\sc InvBasis}.

\section{Gerdt's algorithm}

\label{gerdt}
The Gerdt's algorithm \cite{gerdtnew} improves the above algorithm {\sc InvBasis}. It uses the involutive form of Buchberger's criteria to avoid unnecessary reductions and avoids the repeated processing of non-multiplicative prolongations. In order to explain the structure of this algorithm, we recall some more definitions and notations.

\begin{defn} (\cite{gerdtnew})
An {\em ancestor} of a polynomial $f \in F \subset R \setminus \{0\}$, denoted by $\anc(f)$, is a polynomial $g \in  F$
of the smallest $\deg(\LM(g))$ among those satisfying $\LM(f) = u\LM(g)$ where $u$ is either the unit monomial or a power product of non-multiplicative variables for $\LM(g)$ and such that $\nf_{\L}(f-ug,F\setminus\{f\})=0$ if $f\neq ug$.
\end{defn}

This additional information for an element in a polynomial set is useful to avoid unnecessary reductions specially by applying the adapted Buchberger's criteria (see below).

\begin{prop}{\em (\cite{detecting})}
\label{crit}
Let $I\subset R$ be an ideal and $G\subset I$ be a finite set. Let also $\prec$ be a monomial ordering on $R$ and ${\L}$ an involutive division. Then $G$ is an ${\L}-$involutive basis for $I$ if for all $f\in G$ one of the following conditions holds (we use "$\sqsubset$" to denote proper conventional division):
\begin{enumerate}
\item for all $x\in NM_{\L}(\LM(f),\LM(G))$ the equality $ \nf_{\L}(xf,G)=0$ holds.
\item There exists $g \in G$ with $\LM(g) |_{\L} \LM(f)$ satisfying one of the following:
\begin{description}
\item[$(C_1)$] $\LM(\anc(f))\LM(\anc(g))=\LM(f)$
\item[$(C_2)$] $\lcm(\LM(\anc(f)),\LM(\anc(g))) \sqsubset \LM(f)$
\item[$(C_3)$] There exists $t\in G$ such that:
     \begin{itemize}
         \item $\lcm(\LM(t),\LM(f)) \sqsubset \lcm(\LM(\anc(f)),\LM(\anc(g)))$
         \item $\lcm(\LM(t),\LM(g)) \sqsubset \lcm(\LM(\anc(f)),\LM(\anc(g)))$
     \end{itemize}
\item[$(C_4)$] There exists $t\in G$ computed before $\anc(f)$ and $y\in NM_{\L}(\LM(t),\LM(G))$ s.t.
      \begin{itemize}
          \item $y\LM(t)=\LM(f)$
          \item $\lcm(\LM(\anc(f)),\LM(\anc(t))) \sqsubset \LM(f)$.
      \end{itemize}
\end{description}
\end{enumerate}
\end{prop}

Now we can present Gerdt's algorithm for computing involutive bases in which we associate to each polynomial $f$, the triple $p=\{f,g,V\}$ where $f=\poly(p)$ is the polynomial itself, $g=\anc(p)$ is its ancestor and $V=NM(p)$ is the list of non-multiplicative variables of $f$ which have been already processed in the algorithm.  If $P$ is a set of triples, we denote by $\poly(P)$ the set $\{\poly(p) \ \mid \ p\in P\}$. If no confusion arises, we may refer to a triple $p$ instead of $\poly(p)$, and vice versa.

\begin{algorithm}[H]
\caption{{\sc Gerdt}}
\begin{algorithmic}
 \INPUT  $F$, a set of polynomials; ${\L}$, an involutive division; $\prec$, a monomial ordering
 \OUTPUT a minimal involutive basis for $\li F \ri$ w.r.t ${\L}$ and $\prec$
\STATE Select $f \in F$ with no proper divisor of $\LM(f)$ in $\LM(F)$
\STATE $T:=\{\{f,f,\emptyset\}\}$;
\STATE $Q:=\{\{q,q,\emptyset\}\ |\ q\in F \setminus \{f\}\}$;
\STATE $Q:=${\sc HeadReduce}$(Q,T,{\L},\prec)$;
\WHILE {$Q \ne \emptyset$}
     \STATE Select and remove $p \in Q$ with no proper divisor of $\LM(\poly(p))$ in $\LM(\poly(Q))$;
     \IF {$\poly(p)=\anc(p)$}
           \FOR {$q\in T$ with $\LM(\poly(p))\sqsubset \LM(\poly(q))$}
                 \STATE $Q:=Q \cup \{q\}$;
		 \STATE $T:=T \setminus \{q\}$;
           \ENDFOR
     \ENDIF
      \STATE $h:=${\sc TailNormalForm}$(p,T,{\L},\prec)$
      \STATE $T:=T \cup \{\{h,\anc(p),NM(p)\}\}$;
      \FOR {$q\in T$ and $x\in NM_{\L}(\LM(\poly(q)),\LM(\poly(T))) \setminus NM(q)$}
           \STATE $Q:=Q \cup \{\{x\ \poly(q),\anc(q),\emptyset\}\}$;
           \STATE $NM(q):=NM(q) \cap NM_{\L}(\LM(\poly(q)),\LM(\poly(T))) \cup \{x\}$;
      \ENDFOR
      \STATE $Q:=${\sc HeadReduce}$(Q,T,{\L},\prec)$;
\ENDWHILE
\STATE {\bf Return} ($\poly(T)$)
\end{algorithmic}
\end{algorithm}

This algorithm includes three subalgorithms {\sc HeadReduce}, {\sc HeadNormalForm} and {\sc TailNormalForm} which we present below.

\begin{algorithm}[H]
\caption{{\sc HeadReduce}}
\begin{algorithmic}
 \INPUT $Q$ and $T$, sets of triples; ${\L}$, an involutive division; $\prec$, a monomial ordering
 \OUTPUT The $Q$ of triples whose polynomials are ${\L}$-head reduced modulo $T$
 \STATE $S:=Q$
\STATE $Q:=\emptyset$;
\WHILE {$S \ne \emptyset$}
     \STATE Select and remove $p \in S$
     \STATE $h:=${\sc HeadNormalForm}$(p,T,{\L})$;
     \IF {$h\ne 0$}
          \IF{$\LM(\poly(p))\ne \LM(h)$}
               \STATE $Q:=Q \cup \{\{h,h,\emptyset\}\};$
          \ELSE
                \STATE $Q:=Q \cup \{p\};$
          \ENDIF
     \ELSE
          \IF{$\LM(\poly(p))=\LM(\anc(p))$}
               \STATE $S:=S \setminus \{q\in S \ | \ \anc(q)=\poly(p)\}$;
          \ENDIF
      \ENDIF
\ENDWHILE
\STATE {\bf Return} ($Q$)
\end{algorithmic}
\end{algorithm}

In the below subalgorithm {\sc HeadNormalForm} the Boolean expression Criteria$(p,g)$ is true if at least one of the four criteria in Proposition \ref{crit} holds for $p$ and $g$.

\begin{algorithm}[H]
\caption{{\sc HeadNormalForm}}
\begin{algorithmic}
 \INPUT $T$, a set of triples; $p$, a triple; ${\L}$, an involutive division; $\prec$, a monomial ordering
 \OUTPUT ${\L}$-head normal form of $\poly(p)$ modulo $T$
 \STATE $h:=\poly(p)$;
\STATE $G:=\poly(T)$;
\IF{$\LM(h)$ is ${\L}$-irreducible modulo $G$}
     \STATE {\bf Return} $(h)$
\ELSE
     \STATE Select $g\in G$ with $\LM(\poly(g)) |_{\L} \LM(h)$;
          \IF{$\LM(h)\ne \LM(\anc(p))$}
               \IF{Criteria$(p,g)$}
                     \STATE {\bf Return} $(0)$
               \ENDIF
          \ELSE
               \WHILE {$h\ne 0$ and $\LM(h)$ is ${\L}$-reducible modulo $G$}
                     \STATE Select $g\in G$ with $\LM(g) |_{\L} \LM(h)$;
                     \STATE $h:=h-g\frac{\LT(h)}{\LT(g)}$;
               \ENDWHILE
           \ENDIF
\ENDIF
\STATE {\bf Return} ($h$)
\end{algorithmic}
\end{algorithm}

\begin{algorithm}[H]
\caption{{\sc TailNormalForm}}
\begin{algorithmic}
 \INPUT $T$, a set of triples; $p$, a triple, which is ${\L}$-head reduced modulo T;
 ${\L}$, an involutive division; $\prec$, a monomial ordering
 \OUTPUT ${\L}$-Normal form of $\poly(p)$ modulo $T$
 \STATE $h:=\poly(p)$;
\STATE $G:=\poly(T)$;
\WHILE {$h$ has a term $t$ which is  ${\L}-$reducible modulo $G$}
       \STATE Select $g\in G$ with $\LM(g) |_{\L} t$;
       \STATE $h:=h-g\frac{t}{\LT(g)}$;
\ENDWHILE
\STATE {\bf Return} ($h$)
\end{algorithmic}
\end{algorithm}

\section{Modified algorithm}
\label{sec:3}
In this section, we present the following variant of Gerdt's algorithm for computing involutive bases.

\begin{algorithm}[H]
\caption{{\sc VarGerdt}}
\begin{algorithmic}
\INPUT $F$, a set of polynomials; ${\L}$, an involutive division; $\prec$, a monomial ordering
\OUTPUT a minimal ${\L}-$involutive basis for $\li F \ri$
\STATE Select $f \in F$ with no proper divisor of $\LM(f)$ in $\LM(F)$
\STATE $T:=\{\{f,f,\emptyset\}\}$;
\STATE $Q:=\{\{q,q,\emptyset\}\ |\ q\in F \setminus \{f\}\}$;
\WHILE {$Q \ne \emptyset$}
     \STATE Select and remove $p \in Q$ with no proper divisor of $\LM(\poly(p))$ in $\LM(\poly(Q))$;
     \STATE $h:=${\sc NormalForm}$(p,T,{\L},\prec)$;
     \IF {$h = 0$ and $\LM(\poly(p))=\LM(\anc(p))$}
           \STATE $Q:=\{q\in Q \ |\ \anc(q)\ne \poly(p)\};$
     \ELSE \IF{$\LM(\poly(p))\ne \LM(h)$}
               \FOR {$q\in T$ with $\LM(\poly(h))\sqsubset \LM(\poly(q))$}
                    \STATE $Q:=Q \cup \{q\}$;
		    \STATE $T:=T \setminus \{q\}$;
               \ENDFOR
               \STATE $T:=T \cup \{\{h,h,\emptyset\}\};$
          \ELSE
               \STATE $T:=T \cup \{\{h,\anc(p),NM(p)\}\};$
      \ENDIF
      \FOR {$q\in T$ and $x\in NM_{\L}(\LM(\poly(q)),\LM(\poly(T))) \setminus NM(q)$}
           \STATE $Q:=Q \cup \{\{x\ \poly(q),\anc(q),\emptyset\}\}$;
           \STATE $NM(q):=NM(q) \cap NM_{\L}(\LM(\poly(q)),\LM(\poly(T))) \cup \{x\}$;
      \ENDFOR
       \ENDIF
\ENDWHILE
\STATE {\bf Return} ($\poly(T)$)
\end{algorithmic}
\end{algorithm}

\noindent
This variant seems to be simpler both to understand and to implement. We also present here a proof of its correctness and termination including the correctness proof for the used criteria.

The subalgorithm {\sc NormalForm} returns the ${\L}-$normal form of a polynomial w.r.t. a given set of polynomials (see below). Moreover, this  subalgorithm  detects some unnecessary reductions using the involutive form of Buchberger's criteria. It is worth noting that the first author and Yanovich \cite{yanovich} have shown that the use of $C_3$  and $C_4$ criteria (see Proposition \ref{crit}) may notably slow down the computation of involutive bases. That is why, we use only  $C_1$  and $C_2$ criteria. Below, we give a simple proof of correctness for these criteria.

\begin{lem}
  Let $I\subset R$ be an ideal, $\prec$ a monomial ordering on $R$ and ${\L}$ an involutive division. Let $T\subset I$ be the last computed basis during a computation of an involutive basis for $I$, and $f\in I$. Then $\nf_{\L}(f,T)=0$ if there exists $q \in T$ with $\LM(q) |_{\L} \LM(f)$  satisfying one of the following conditions:
 \vskip 0.1cm
\begin{enumerate}
\item[$(C_1)$] $\LM(\anc(f))\LM(\anc(q))=\LM(f)$
\item[$(C_2)$] $\lcm(\LM(\anc(f)),\LM(\anc(q))) \sqsubset \LM(f)$.
\end{enumerate}
\end{lem}
\begin{proof}
  Suppose that $(f,q)$ satisfies $C_1$. Then, for $\lcm(\LM(\anc(f)),\LM(\anc(q)))$ two cases are possible: If it is a proper divisor of $\LM(\anc(f))\LM(\anc(q))=\LM(f)$, there exists a monomial $s\ne 1$ such that $\LM(f)=s\lcm(\LM(\anc(f)),\LM(\anc(q)))$. Let $f=u\cdot\anc(f)$, $q=v\cdot \anc(q)$ and $\LT(f)=t\cdot \LT(q)$ for some monomials $u$ and $v$ and some term $t$. Thus, we can write $f-tq=u\cdot \anc(f)-tv\cdot \anc(q)=s(u'\cdot \anc(f)-v'\cdot \anc(q))$ where $u=su'$ and $vt=sv'$ for some monomials $u',v'$. Since, $\lcm(\LM(\anc(f)),\LM(\anc(q)))$  is a proper divisor of $\LM(f)$, then $u'\cdot \anc(f)-v'\cdot \anc(q)$ has been computed before $f$ (see the selection strategy for  choosing polynomials in the algorithm), and therefore, it has a standard representation w.r.t. $T$. This implies that $f-tq$ has also a standard representation w.r.t. $T$, and $\nf_{\L}(f,T)=0$. As the second case, if $\lcm(\LM(\anc(f)),\LM(\anc(q)))$ is equal to $\LM(\anc(f))\LM(\anc(q))$, then by the Buchberger's first criterion, we can conclude that $u'\cdot \anc(f)-v'\cdot \anc(q)$ has a standard representation w.r.t. $T$, and this proves the assertion like the first case.

Now, if $(f,q)$ satisfies $C_2$, the proof is similar to that for the above first case.
\end{proof}

Just as above, for two polynomials $p$ and $q$, the Boolean expression Criteria($p,q$) is true if they satisfy at least one of the criteria $C_1$ or $C_2$, and false otherwise.

\begin{algorithm}[H]
\caption{{\sc NormalForm}}
\begin{algorithmic}
 \INPUT $p$, a triple; $T$, a set of triples; ${\L}$, an involutive division; $\prec$, a monomial ordering
 \OUTPUT ${\L}$-Normal form of $p$ modulo $T$
 \STATE $h:=\poly(p)$;
 \STATE $G:=\poly(T)$;
 \WHILE {$h$ has a term $t$ which is  ${\L}-$reducible modulo $G$}
       \STATE Select $g\in G$ with $\LM(g) |_{\L} t$;
       \IF {$t=\LT(\poly(p))$ and Criteria($h,g$)}
              \STATE {\bf Return} ($0$)
       \ELSE
              \STATE $h:=h-g\frac{t}{\LT(g)}$;
       \ENDIF
\ENDWHILE
\STATE {\bf Return} ($h$)
\end{algorithmic}
\end{algorithm}

\begin{thm}
  The {\sc VarGerdt} algorithm computes a minimal involutive basis for the ideal generated by the input polynomial set.
\end{thm}
\begin{proof}
  Let $F\subset R$ be a finite set of polynomials taken as the input of {\sc VarGerdt} algorithm. Let $\prec$ be a monomial ordering on $R$ and ${\L}$ be a constructive noetherian involutive division \cite{gerdt0}. Let $G$ be a a polynomial set that is the output of the algorithm for $F$. Since (by the structure of the algorithm) for each $f\in G$ and each $x\in NM_{\L}(\LM(f),\LM(G))$, we examine $xf$ (we either compute its ${\L}-$normal form w.r.t. $G$ or we discard it by the criteria), it follows that $\nf_{\L}(xf,G)=0$. Therefore, $G$ is an involutive basis for $\li F \ri$ by Theorem \ref{blin}. Moreover, this basis is minimal by the definition of this concept due to  the structure of the {\sc VarGerdt} algorithm and constructivity of ${\L}$.

The termination of {\sc VarGerdt} algorithm follows from the noetherianity of ${\L}$, see Proposition \ref{jan}.
\end{proof}
\begin{rem}
  The algorithm {\sc VarGerdt} as well as Gerdt's algorithm not necessarily outputs an involutively tail autoreduced basis. If the output is $G$, then, in accordance to the definition that follows Theorem 6 of Section 2, the tail autoreduction is provided by the command
\[
   (\,\forall g\in G\,)\ \ [\,g:=NF_{{\L}}(g,G\setminus \{g\})\,]\,.
\]
\end{rem}

\section{Experiments and results}
\label{Exp}
We have implemented both algorithms {\sc  VarGerdt} and {\sc Gerdt} in {\tt Maple} 14\footnotemark\footnotetext{The {\tt Maple} code of our programs and examples is available on the Web page {\tt http://invo.jinr.ru/}}. For an efficient implementation of {\sc Gerdt} algorithm, we refer to \cite{blinkov}. It is worth noting that, in this paper, we are willing to compare the structure and behavior of {\sc  VarGerdt} and {\sc Gerdt} algorithms on the same platform. Therefore, we do not compare our implementations with \cite{blinkov}.

To compare the behavior of these algorithms, we used some well-known examples from the the collection \cite{Bini} have been widely used for verification and comparison of different software packages for
construction of \Gr bases. In our benchmarking with these examples we applied the following selection strategy for the polynomials in $Q$. An element at the initialization step to be inserted in set $Q$ as well as that selected from $Q$ in the main loop has the lowest leading monomial with respect to the degree-reverse-lexicographic ordering. In the case of several such polynomials  that is selected whose ancestor has been examined earlier than the ancestors of the other polynomials with the identical leading monomials.

 The results are shown in the following tables where computation was performed on a personal computer  with  2.33 GHz, 2$\times$Intel(R) Xeon(R) Quad core, 16 GB RAM and $64$ bites under the Linux operating system. The computation was done over $\mathbb{Q}$ and the monomial ordering is always the degree-reverse-lexicographical one.

 \begin{center}
\begin{figure}[H]
\centering {\tiny
\begin{tabular}{|c||c|c|c|c|c|c|c|}
\cline{1-8}
Noon$4$  & time  & memo. & reds.  & $C_1$ & $C_2$ &R & deg.\\
\cline{1-8}
\noalign{\vskip-10pt} & & & & & & &\\
{\sc VarGerdt}  & 14.29 & 819.2 & 56  & 6 & 19 & 0 &10\\
\cline{1-8}
{\sc Gerdt}  &   14.92 & 821.4 & 50  & 6 & 19 & 0 &9 \\
\cline{1-8}
\multicolumn{1}{c}{}  \\

\cline{1-8}
Cyclic$5$  & time  & memo. & reds.  & $C_1$ & $C_2$ &R & deg.\\
\cline{1-8}
\noalign{\vskip-10pt} & & & & & & &\\
{\sc VarGerdt}  & 12.15 & 701.8 & 77  & 36 & 9 & 18 &9\\
\cline{1-8}
{\sc Gerdt}  &   12.80 & 686.6 & 83  & 40 & 5 & 0 &8 \\
\cline{1-8}
\multicolumn{1}{c}{}  \\

\cline{1-8}
Sturmfels  & time  & memo. & reds.  & $C_1$ & $C_2$ &R & deg.\\
\cline{1-8}
\noalign{\vskip-10pt} & & & & & & &\\
{\sc VarGerdt}  & 64.29 & 4152.6 & 91  & 43 & 212 & 0 &6\\
\cline{1-8}
{\sc Gerdt}  &   80.72 & 5103.7 & 80  & 42 & 218 & 0 &5 \\
\cline{1-8}
\multicolumn{1}{c}{}  \\

\cline{1-8}
Katsura$5$  & time  & memo. & reds.  & $C_1$ & $C_2$ &R & deg.\\
\cline{1-8}
\noalign{\vskip-10pt} & & & & & & &\\
{\sc VarGerdt}  & 97.18 & 9636.8 & 47  & 22 & 1 & 0 &12\\
\cline{1-8}
{\sc Gerdt}  &   54.50 & 4731.8 & 47  & 22 & 1 & 0 &10 \\
\cline{1-8}
\multicolumn{1}{c}{}  \\

\cline{1-8}
Eco$7$  & time  & memo. & reds.  & $C_1$ & $C_2$ &R & deg.\\
\cline{1-8}
\noalign{\vskip-10pt} & & & & & & &\\
{\sc VarGerdt}  & 102.62 & 6066.8 & 132  & 51 & 48 &20 &6\\
\cline{1-8}
{\sc Gerdt}  &   80.21 & 4249.7 & 124  & 46 & 40 & 0 &4 \\
\cline{1-8}
\multicolumn{1}{c}{}  \\

\cline{1-8}
Liu  & time  & memo. & reds.  & $C_1$ & $C_2$ &R & deg.\\
\cline{1-8}
\noalign{\vskip-10pt} & & & & & & &\\
{\sc VarGerdt}  & 2.65 & 155.1 & 18  & 6 & 3 & 0 &0\\
\cline{1-8}
{\sc Gerdt}  &   3.81 & 236.0 & 18  & 6 & 3 & 0 &0 \\
\cline{1-8}
\multicolumn{1}{c}{}  \\

\cline{1-8}
Lichtblau  & time  & memo. & reds.  & $C_1$ & $C_2$ &R & deg.\\
\cline{1-8}
\noalign{\vskip-10pt} & & & & & & &\\
{\sc VarGerdt}  & 111.21 & 22302.9 & 25  & 0 & 9 & 1 &11\\
\cline{1-8}
{\sc Gerdt}  &   $> 8$ hours & ? & ?  & ? & ? & ? &? \\
\cline{1-8}
\multicolumn{1}{c}{}  \\

\cline{1-8}
Katsura$6$  & time  & memo. & reds.  & $C_1$ & $C_2$ &R & deg.\\
\cline{1-8}
\noalign{\vskip-10pt} & & & & & & &\\
{\sc VarGerdt}  & 213.59 & 16124.3 & 128  & 44 & 3 & 0 &8\\
\cline{1-8}
{\sc Gerdt}  &   4227.32 & 1632783.6 & 128  & 44 & 3 & 0 &7 \\
\cline{1-8}
\multicolumn{1}{c}{}  \\

\cline{1-8}
Cyclic$6$  & time  & memo. & reds.  & $C_1$ & $C_2$ &R & deg.\\
\cline{1-8}
\noalign{\vskip-10pt} & & & & & & &\\
{\sc VarGerdt}  & 1313.88 & 222625.1 & 542  & 169 & 7 & 47 &11\\
\cline{1-8}
{\sc Gerdt}  &   2571.84 & 655916.5 & 476  & 152 & 18 & 0 &10 \\
\cline{1-8}
\multicolumn{1}{c}{}  \\

\end{tabular}}
\end{figure}
 \end{center}

The time (resp. memo., reds., $C_1$ and $C_2$) column shows the consumed CPU time in seconds (resp. amount of megabytes of memory used, number of reductions to zero,  the number of polynomials removed by $C_1$ and $C_2$ criteria) by the corresponding algorithm. The seventh column indicates the number of polynomials eliminated by {\em Rewritten} criterion: in {\sc  VarGerdt} algorithm, as well as in {\sc Gerdt} algorithm, when the  polynomial part of a triple $\{f,g,V\}$ reduces to zero, and $f=g$, then we can discard any polynomial whose ancestor is $f$. We call this criterion the Rewritten criterion. The last column shows the largest degree of intermediate polynomials treated during the computation of involutive bases.

A comparison of the timing columns in the above tables and our test for some other examples  shows  that {\sc VarGerdt} is more efficient and stable than {\sc Gerdt}. Indeed, the main difference between these algorithms is that in {\sc Gerdt} algorithm, we do first the head reduction of all non-multiplicative prolongations (which have not been examined yet) and then we choose one of them to complete its involutive reduction to the full normal form, while in {\sc VarGerdt} we choose a polynomial (among the non-multiplicative prolongations), and we compute its full normal form.

In Gerdt's algorithm the head involutive reduction of the whole set of non-multiplicative prolongation is done in order to provide a platform for matching a good selection strategy. We refer to paper \cite{GB07} for details on heuristically good selection strategies for the Janet division. The choice of such selection strategy as well as the use of proper data structures for fast search of involutive divisor \cite{gerdtnew} plays a key role in providing a high computational efficiency of the involutive algorithms. In this paper we do not consider these important aspects of the practical construction of involutive bases and compare two algorithms experimentally for their simplest and identical for the both algorithms implementation.

\section*{Acknowledgements.}  The research presented in the paper was completed during the stay of the second author (A. Hashemi) at the Joint Institute for Nuclear Research in Dubna, Russia. He would like to thank Professor V. P. Gerdt  for the invitation, hospitality and support. The contribution of the first author (V .P. Gerdt) was partially supported by by grant 01-01-00200 from the Russian Foundation for Basic Research and by grant 3810.2010.2 from the Ministry of Education and Science of the Russian Federation.

\bibliographystyle{alpha}

\end{document}